\documentclass[10pt]{article}
\usepackage{amssymb}
\usepackage{graphicx}
\usepackage{xcolor} 
\usepackage{tensor}
\usepackage{fullpage} 
\usepackage{amsmath}
\usepackage{amsthm}
\usepackage{verbatim}
\usepackage{hyperref}
\usepackage{cite}
\usepackage{array}

\usepackage{enumitem}
\setlist[enumerate]{leftmargin=1.2em}
\setlist[itemize]{leftmargin=1.2em}

\usepackage{seqsplit,mathtools}

\definecolor{green}{rgb}{0,0.8,0} 

\newtheorem{theorem}{Theorem}[section]

\newtheorem{lemma}[theorem]{Lemma}

\theoremstyle{definition}
\newtheorem{definition}[theorem]{Definition}

\theoremstyle{remark}
\newtheorem{remark}[theorem]{Remark}

\numberwithin{equation}{section}
\newcommand{\nrm}[1]{\Vert#1\Vert}

\newcommand{\brk}[1]{\langle#1\rangle}

\newcommand{\nnrm}[1]{{\vert\kern-0.25ex\vert\kern-0.25ex\vert #1 
		\vert\kern-0.25ex\vert\kern-0.25ex\vert}}

\newcommand{\supp}{{\mathrm{supp}}\,}

\renewcommand{\Re}{\mathrm{Re}}

\newcommand{\lap}{\Delta}

\newcommand{\rd}{\partial}
\newcommand{\nb}{\nabla}

\newcommand{\whf}{\widehat{f}}

\newcommand{\alp}{\alpha}
\newcommand{\bt}{\beta}

\newcommand{\eps}{\epsilon}

\newcommand{\kpp}{\kappa}

\newcommand{\Lmb}{\Lambda}

\newcommand{\tht}{\theta}

\newcommand{\bbR}{\mathbb R}

\newcommand{\calC}{\mathcal C}

\newcommand{\R}{\mathbb{R}}

\newcommand{\bq}{\begin{equation}}
\newcommand{\eq}{\end{equation}}
\newcommand{\e}{\varepsilon}
\newcommand{\lt}{\left}
\newcommand{\rt}{\right}

\newcommand{\intr}{\int}
\newcommand{\intor}{\iint}
\newcommand{\intrr}{\iint}

\begin{document}
\bibliographystyle{plain}
 \title{Well-posedness and singularity formation for Vlasov--Riesz system}
\author{Young-Pil Choi\thanks{Department of Mathematics, Yonsei University, Seoul 03722, Republic of Korea. E-mail: ypchoi@yonsei.ac.kr} \and In-Jee Jeong\thanks{Department of Mathematics and RIM, Seoul National University, Seoul 08826. E-mail: injee\_j@snu.ac.kr}   }

\date\today
 
  \maketitle

\renewcommand{\thefootnote}{\fnsymbol{footnote}}
\renewcommand{\thefootnote}{\arabic{footnote}}

\begin{abstract} 
	We investigate the Cauchy problem for the Vlasov--Riesz system, which is a Vlasov equation featuring an interaction potential generalizing previously studied cases, including the Coulomb $\Phi = (-\lap)^{-1}\rho$, Manev $(-\lap)^{-1} + (-\lap)^{-\frac12}$, and pure Manev $(-\lap)^{-\frac12}$ potentials. For the first time, we extend the local theory of classical solutions to potentials more singular than that for the Manev. Then, we obtain finite-time singularity formation for solutions with various attractive interaction potentials, extending the well-known blow-up result for attractive Vlasov--Poisson for $d\ge4$. Our local well-posedness and singularity formation results extend to cases when linear diffusion and damping in velocity are present. 
\end{abstract}

\section{Introduction}
In this paper, we study the initial value problem for the \textit{Vlasov--Riesz} system and its generalizations, which are Vlasov equations with general interaction forces, featuring the Coulomb and Manev potential as special cases. To be more specific, let $f = f(t,x,v)$ be the number density of particles at position $x \in \R^d$ with velocity $v \in \R^d$ at time $t$. Then, our main system reads as
\begin{equation} \tag{VR}\label{eq:VR}
\left\{
\begin{aligned}
&\rd_t f + v\cdot\nb_x f + \nb \Phi \cdot \nb_v f = \sigma \nb_v \cdot (\nb_v f + vf),\quad (x,v) \in \R^d \times \R^d, \quad t > 0, \\
&\Phi = \kappa\Lmb^{-\beta} \rho, \quad \rho = \intr f \, dv,
\end{aligned}
\right.
\end{equation} 
Here, $\Lmb^{-\bt}$ is defined by the Fourier multiplier with symbol $|\xi|^{-\bt}$, where $\xi$ is the dual variable for $x$. The cases $\kappa > 0$ and $\kappa < 0$ correspond to the attractive and repulsive interactions, respectively. The right hand side of \eqref{eq:VR} is the linear Fokker--Planck operator, where the coefficient $\sigma$ is nonnegative. 

The system \eqref{eq:VR} covers not only the classical Coulomb interaction $\beta = 2$, i.e. $\Phi = (-\Delta)^{-1}\rho$ but also the so-called {\it pure Manev} or {\it Manev correction} \cite{BDIV97} given by $\Phi = \Lmb^{-1}\rho$. We shall also consider potentials given by linear combination of $\Lmb^{-\bt}$, the principal example being the {\it Manev} potential $\Phi = \Lmb^{-1}\rho + (-\Delta)^{-1}\rho$ proposed by Manev in \cite{Man1, Man2, Man4, Man3} as a modification of the Newtonian gravitation law. Apart from the Einstein's theory of general relativity, the additional term $\Lmb^{-1}\rho$ allows to explain various phenomena observed in the solar system, such as the anomalous secular precession of Mercury's perihelion, gravitational redshift, free gravitational collapse, and so on; see \cite{Hall94,H07, Man1, Man2, Man4, Man3, Ure98, BDIV97, Ure99, Kirk}. It is interesting to note that Sir Isaac Newton himself used a Manev-type gravitational potential to explain the dynamics of the Moon (\cite[Book I,
Section IX, Proposition XLIV, Theorem XIV]{Newton}). Moreover, Hall in \cite{Hall94} suggests an interaction potential which is \textit{even more singular} than that of Manev, to deal with Mercury's precession. We refer to \cite{DMMS95,H07, Ill00,IP05, SSM99, Diacu} and references therein for a detailed discussion of the history and applications of the modified Newtonian's law including the Manev's nonrelativistic gravitational law. Then, the Riesz interaction $\Lmb^{-\beta}\rho$ can be considered as a generalization, and such interaction has been indeed studied in the physics literature, e.g. \cite{BBDR05, Maz11, Tor16}. It also appears in the study of equilibrium properties of a system of point particles interacting via Coulomb or Riesz interactions and confined by an external potential \cite{LS17, PR18, RS16}.

\subsection{Main Results}

The main purpose of this work is to develop a local well-posedness theory of strong solutions for the system \eqref{eq:VR} which is applicable even in the singular regime $\bt<1$ and establish finite-time singularity formation within our local well-posedness framework. Our main results cover both $\sigma=0$ and $\sigma>0$, and it seems that there have been no results on singularity formation in the latter case, namely when the Fokker--Planck term is present. 

\subsubsection{Local well-posedness}

To put our well-posedness result in context, let us briefly review the classical results on the initial value problem for \eqref{eq:VR}. To begin with, the case $\bt=2$ which corresponds to the famous Vlasov--Poisson ($\sigma=0$) or Vlasov--Poisson--Fokker--Planck system ($\sigma>0$) has been studied extensively. The global existence of weak solutions are discussed in \cite{CS95, LP91, Pal12, Vic91}, and classical solutions (which are global when $d\le3$) are obtained in \cite{Bou93,BD85, Degond86,Pfa92, UO78, VO90}. The asymptotic behavior of solutions is studied in \cite{BD95, CSV95, Her07, HJ13} and \cite{Gla96, Pert04, Rein07} are a few standard references for general theory on the Vlasov--Poisson system and related kinetic equations. In the case $\beta=1$ and $\sigma=0$, i.e. Vlasov--Manev system where the interaction potential is more singular by order 1, the local-in-time existence and uniqueness of classical solutions are obtained in \cite{IVDB98}, using the method of characteristics. 
Although it is not clearly stated in \cite{IVDB98}, the arguments therein can be easily adapted to produce local-in-time unique classical solution to the system \eqref{eq:VR} when $\beta \in [1,2]$, namely the cases ``interpolating'' between Coulomb and Manev type potentials.

In the case $\beta \ge 1$, a standard energy estimate in Sobolev spaces $H^s$ can be employed to prove local well-posedness, for large $s$. This seems to be sharp, since when $\bt=1$, the advecting velocity $\nb\Phi \sim \nb\Lmb^{-1}\rho$ has the same Sobolev regularity with $f$ in the $x$ variable for each $t$. While one may expect ill-posedness in the more singular regime $\beta<1$, our main result shows that local well-posedness for smooth solutions still persists as long as $\beta>\frac34$. The key observation is that the \textit{velocity averaging} effect, coming from the kinetic transport $\rd_t + v\cdot\nb_x$, improves regularity of $\rho$ compared with $f$ when suitably \textit{averaged in time}. While such a kinetic averaging lemma is widely used in the study of global-in-time regularity of solutions to kinetic transport equations, see e.g. \cite{GLPS88, GPS85, DLM91, Vil02, Jab09}, we are not aware of any previous applications of the averaging lemma in the construction of local smooth solutions for systems with singular advection.


For the averaging lemma to be applicable, we shall need to introduce the following weighted Sobolev norm: with $\brk{v}^2 = 1 + |v|^2$ and an integer $N\ge0$, we define $H^{s,2N}_{x,v}(\bbR^d\times\bbR^d)$--norm by \begin{equation*}
	\begin{split}
		\nrm{f}_{H^{s,2N}_{x,v}}^2 & :=  \intrr |\brk{v}^{2N}f|^2 +  | \brk{v}^{2N} \Lmb^s_x f |^2  + | \brk{v}^{2N} \Lmb^s_v f |^2 \, dvdx . 
	\end{split}
\end{equation*} When $N = 0$, we recover the usual Sobolev space $H^s_{x,v} = H^{s,0}_{x,v}$. We are now ready to state our main result. 

\begin{theorem}\label{thm:lwp}
	Assume that $\frac{3}{4}< \beta \le 1$ and $\kpp \in \bbR$. For any $\sigma\ge0$, the Cauchy problem for \eqref{eq:VR} is locally well-posed in $H^{s,2N}_{x,v}(\bbR^d\times \bbR^d)$ for any real $s > \frac{d}{2}+1$ and integer $N>\frac{d}{4}$. That is, for any initial data $f_0 \in  H^{s,2N}_{x,v}(\bbR^d\times \bbR^d)$, there exist $T>0$ and a unique solution $f$ to \eqref{eq:VR} with $f(t=0) = f_0$ belonging to $L^\infty([0,T]; H^{s,2N}_{x,v}(\bbR^d\times \bbR^d))$. 
	
	When $\sigma = 0$ and $f_0 \in H^{s}_{x,v}(\bbR^d\times \bbR^d)$ is compactly supported in $v$, there exist $T>0$ and a unique corresponding solution in $L^\infty([0,T]; H^{s}_{x,v}(\bbR^d\times \bbR^d))$ which is compactly supported in $v$ for each $t \in [0,T]$. 
	
\end{theorem}

\begin{remark} Note that the sign of $\kpp$ is irrelevant in the above. Moreover, Theorem \ref{thm:lwp} readily extends to systems where the operator $\Lmb^{-\bt}$ in \eqref{eq:VR} is replaced with convolution-type operators which have the same order of singularity in the kernel. For instance, one may treat $\Lmb^{-\bt} + \sum_{i} c_i \Lmb^{-\bt_i}$ with any $c_i\in\bbR$ and $\bt_i>\bt$. 
\end{remark}

\subsubsection{Finite-time singularity formation}

Our second main result is the finite-time loss of smoothness of solutions to the equation \eqref{eq:VR}. It is a well-known result of Horst \cite{Horst} (also see \cite[Section 4.6]{Gla96}) that when $d\ge4$, any smooth solution to the Vlasov--Poisson system, i.e. \eqref{eq:VR} with $\beta = 2$ and $\sigma=0$, in the gravitational case can exist only on a finite interval of time. For the Vlasov--(pure) Manev system, i.e. \eqref{eq:VR} with $\beta = 1$ and $\sigma=0$, singularity formation is shown in \cite{BDIV97}. We extend these classical results to the case with attractive singular power-law potential, even when the Fokker--Planck term is present. In particular, our result shows the finite time loss of smoothness of solutions to the Vlasov--Poisson--Fokker--Planck equations in the attractive case under suitable assumptions on the initial data, for high dimensions. Later, we shall also prove singularity formation in certain cases when both the attractive and repulsive potentials are present. 

Although our well-posedness theory mainly focuses on the case $\beta \in (\frac34 , 1]$, our second result covers more general interaction potentials. Thus for the analysis of the finite-time singularity formation, it will be convenient to use the notation $K\star\rho$ in place of $\Phi=\Lmb^{-\bt}\rho$ and consider the following general form of $K$:
\[
K(x) = \sum_{i=1}^N c_i K_i(x), \quad K_i(x) = \frac1{|x|^{\alpha_i}},
\]
where $c_i > 0$ and $\alpha_i \in (0,d)$ for all $i=1,\dots,N$. Note that for $\beta \in (0,2)$, $\Lmb^{-\beta}\rho$ equals $K \star\rho$ up to an absolute constant, where
\[
K(x) = \frac{1}{|x|^{d-\beta}}.
\]
When $d\ge3$, the Coulomb case corresponds to $\bt=2$. For $d\ge2$, the case $\beta \in (0,2)$ corresponds to the Riesz potential, with the case $\bt=1$ corresponding to the pure Manev potential. 

Since our strategy for singularity formation relies on the energy-type estimates on several physical quantities, including the total energy, radial-weighted momentum, and momentum of inertia, the proof requires a little bit of integrability of solutions. (However, the initial data does not need to have compact support or contain vacuum in any finite regions.) For this purpose, let us introduce a solution space $X$ as follows.
\begin{definition}\label{def_sol} For a given $T>0$, we call $f \in X(T)$ if $f$ is a classical solution to the Cauchy problem \eqref{eq:VR} on the time interval $[0,T]$ satisfying the following conditions of decay at far fields: \begin{equation}\label{eq:decay-assumptions}
		\begin{split}
				f (|v| |x|^2 + (|v| + |x|)|v|^2 + |\ln f| + |\nabla K \star \rho| ) \to 0 
		\end{split}
	\end{equation}
	as $|x|, |v| \to +\infty$ for all $t \in [0,T]$.
\end{definition}
The total energy $E=E(t)$ and momentum of inertia $I=I(t)$ are defined as follows: 
\[
E:= \frac12\intrr |v|^2 f\,dxdv + \intrr f \ln f\,dxdv - \frac12 \intr \rho K\star\rho\,dx
\]
and
\[
I := \frac12\intrr |x|^2 f\,dxdv.
\]
The decay conditions appeared in Definition \ref{def_sol} enable us to estimate the time derivative of the above functions. It is not difficult to show that, under the framework of Theorem \ref{thm:lwp}, \eqref{eq:decay-assumptions} is satisfied for the unique local-in-time solution once we assume enough decay in $x,v$ on the initial data. Let us now state our second main result.

\begin{theorem}\label{thm:blow} Let $T>0$ and $d \geq 3$, and let $f$ be a solution to the system \eqref{eq:VR} satisfying $f \in X(T)$. Suppose $\max_{i=1,\dots,N}\alpha_i > 2$.
\begin{itemize}
\item (Vlasov equation: $\sigma =0$) Suppose the initial kinetic energy is small enough compared to the initial interaction energy:
\[
\intrr |v|^2 f_0 \,dxdv + \lt( \sum_{\alpha_i < 2}c_i \lt(1 - \frac{\alpha_i}2\rt)\rt)^{1 - \frac{2}{\alpha_M - 2}} \lt(c_1 \lt(\frac{\alpha_M}2 - 1\rt)\rt)^{\frac2{\alpha_M-2}} < \intr \rho_0 K\star\rho_0 \,dx,
\]
where $\alpha_M:=\max_{i=1,\dots,N}\alpha_i$.
\item (Vlasov--Fokker--Planck equation: $\sigma>0$) Suppose that the initial total energy and radial-weighted momentum is sufficiently small compared to the initial momentum of inertia so that
\[
2(1+\delta)E(0) + \beta I'(0) < -\beta(\sigma + \beta) I(0) - C_\delta - C_0
\]
for some $\delta > 0$ satisfying $2(1+\delta) < \max_{i=1,\dots, N} \alpha_i$, where $I'(0) := I'(t)|_{t=0}$, i.e. 
\[
I'(0) = \intrr (x\cdot v) f_0\,dxdv,
\]
and the constants $C_\delta$, $\beta$, and $C_0$ are explicitly given as
\[
C_\delta := 4(1+\delta)(1 + \delta^{-1})^{\frac d{2+d}} (e^{-1} 2^{3d} \pi^{2d})^{\frac1{2+d}}, \qquad \beta := \frac{-\sigma + \sqrt{\sigma^2 + 4C_\delta}}{2},
\] 
and
\[
C_0 := \lt( \sum_{\alpha_i < 2(1+\delta)}c_i \lt(1 + \delta - \frac{\alpha_i}2\rt)\rt)^{1 - \frac{2}{\alpha_M - 2}} \lt(c_1 \lt(\frac{\alpha_M}2 - 1 - \delta\rt)\rt)^{\frac2{\alpha_M-2}},
\]
respectively. 
\end{itemize}
Then the life-span of the strong solution is finite.
\end{theorem}

\begin{figure}[h]
	\centering
	\includegraphics[scale=1]{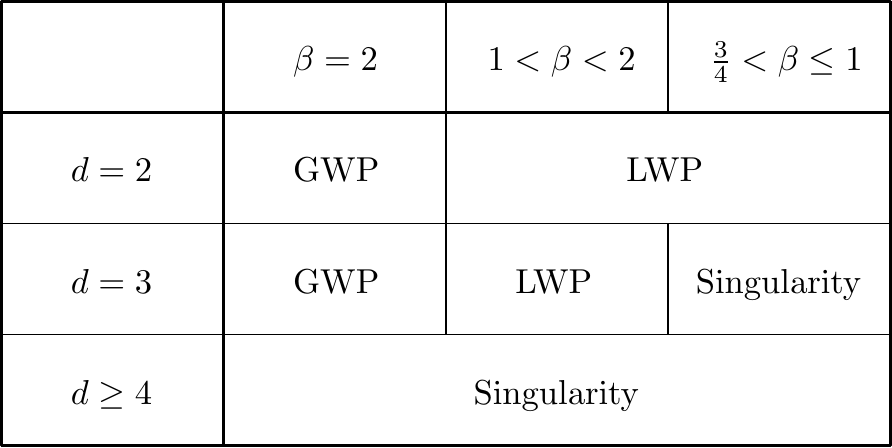}\label{fig:table}
	\caption{Well-posedness and singularity formation for strong solutions of \eqref{eq:VR} in the case $\sigma=0$ and $\kpp>0$}
\end{figure}

See Figure 1 for a simple illustration of the current status of well-posedness and singularity formation of \eqref{eq:VR} in the case $\sigma=0$ and $\kpp>0$. To the best of our knowledge, it is not known whether smooth solutions exist globally in time as soon as the interaction potential becomes more singular than Coulomb in $d=2$. The situation seems to be the same for $1<\bt<2$ when $d=3$. 
\begin{remark} 
	We give several remarks on the statement of Theorem \ref{thm:blow}. 
	\begin{itemize}
		\item It is not difficult to find a class of initial data which satisfies the above sufficient conditions for singularity formation: one can consider ``concentrated'' initial data \begin{equation*}
			\begin{split}
				f_{0}(x,v) = \frac{1}{(\eps')^d} \varphi\left( \frac{x}{\eps'} \right) \frac{1}{\eps^d} \psi\left( \frac{v}{\eps} \right)  
			\end{split}
		\end{equation*} with small $\eps',\eps>0$ and compactly supported smooth bump functions $\varphi,\psi\ge0$. Then, in the limit $\eps',\eps\to0$, we have $I, I' \to 0$ while $E\to-\infty$. 
	\item In the biharmoic case, i.e., $K$ satisfying $(-\Delta)^2 K = \delta_0$ in the sense of distributions, we obtain finite-time singularity formation under the assumptions of Theorem \ref{thm:blow} when $d>6$. 
	\item When $\sigma = 0$, we also have the finite-time singularity formation when $\min_{i=1,\dots,N} \alpha_i \geq 2$ and
	\bq\label{manev}
	\intrr |v|^2 f_0 \,dxdv < \intr \rho_0 K\star\rho_0 \,dx.
	\eq
	In particular, if $K$ is given as the pure Manev potential \cite{BDIV97} in $d=3$:
	\[
	K(x) = \frac{c_1}{|x|^2}, \quad c_1 > 0, 
	\]
	then classical solutions $f \in X(T)$ with initial data satisfying \eqref{manev} cannot exist globally in time. Note that our well-posedness theory covers the regime $\alpha_i \in [2, \frac94)$ in the three-dimensional case. 
	\item Later in Section \ref{subsec:mixed}, we prove singularity formation with potentials having both repulsive and attractive terms. 
	\end{itemize}
\end{remark}




%
%
%
%
%
%
%
%
%
\section{Local well-posedness}\label{sec_well}

This section is devoted to the proof of Theorem \ref{thm:lwp}. First, we review the classical kinetic averaging lemma in Section \ref{subsec:avg}. Then, we establish a priori estimates in Sobolev spaces in Sections \ref{subsec:apriori-cpt} and \ref{subsec:apriori-moment}. We conclude Theorem \ref{thm:lwp} in Section \ref{subsec:eandu} by proving existence and uniqueness of solutions. In what follows, we shall take $\kpp = \pm 1$ for simplicity and write $u(t,x) = \nb\Phi (t,x) = \kpp \nb \Lmb^{-\bt}\rho (t,x)$.

\subsection{Averaging lemma} \label{subsec:avg}

\begin{lemma}[{{\cite[Theorem 7.2]{Gla96}}}] \label{lem:avg}
	Let $h \in L^2_{t,x,v}$ be a solution of 
	\begin{equation*}
		\begin{split}
			\rd_t h + v\cdot \nb_{x} h = \sum_{|\alp|\le m} D_{v}^{\alp} g_{\alp},
		\end{split}
	\end{equation*} and assume that $h$ and $g_{\alp}$ are supported in $v \in B(0,R)$. Here $m \in \mathbb{N}$. Then, for any $\psi \in C^{\infty}_{c}(\bbR^{d})$ compactly supported in $B(0,1)$, we have for any $R\ge 1$ 
	\begin{equation*}
		\begin{split}
			\int_{B(0,R)} h(\cdot,\cdot,v) \psi(R^{-1}v) \, dv  \in H^{s}(\bbR\times\bbR^{d}), \quad s = \frac{1}{2(1+m)}
		\end{split}
	\end{equation*} with \begin{equation}\label{eq:avg-est}
		\begin{split}
			\lt\|	\int_{B(0,R)} h(\cdot,\cdot,v) \psi(R^{-1}v) \, dv  \rt\|_{H^{s}_{t,x}} \le C R^{\frac{d}{2}} \left( \nrm{h}_{L^2_{t,x,v}} + \sum_{|\alp|\le m}\nrm{g_{\alp}}_{L^2_{t,x,v}} \right). 
		\end{split}
	\end{equation} Here, $C>0$ is a constant depending on $m$ and $\psi$ but not on $R$. 
\end{lemma}  
\begin{remark}
	Strictly speaking, \cite[Theorem 7.2]{Gla96} proves a version of \eqref{eq:avg-est} in which dependency in $R$ is not given explicitly. However, following their proof gives $C R^{\frac{d}{2}} $ with $C$ independent of $R$, assuming $R\ge1$.
\end{remark}

\subsection{A priori estimates: compactly supported case} \label{subsec:apriori-cpt}

As a warm-up, we shall first prove the Sobolev a priori estimates in the compactly supported case, assuming $\sigma = 0$. In this case, we can simply use the standard Sobolev spaces $H^s(\bbR^d\times\bbR^d)$. To be precise, we will prove the following result.
\begin{lemma}\label{lem:apriori-cptspt}
	Assume that $\sigma=0$, $\frac34 < \bt \le 1$, $\kpp \in \{ \pm 1 \}$, and $s>\frac{d}{2}+1$. Let $f$ be a sufficiently smooth and fast decaying solution to \eqref{eq:VR} on a time interval $[0,T]$. Furthermore, assume that for some $R_0>0$, $\supp_{v}( f(t=0) ) \subset B_{0}(R_{0})$, where $B_{0}(R_{0}) \subset \bbR^d_v$ is the ball of radius $R_0$ centered at the origin. Then, by taking $T>0$ smaller if necessary, but in a way depending only on $\nrm{f(t=0)}_{H^{s}(\bbR^d\times\bbR^d)}$ and $R_{0}$, we have for all $0<t\le T$ \begin{equation*}
		\begin{split}
			\nrm{f(t)}_{H^{s}(\bbR^d\times\bbR^d)} \le 4 \nrm{f(t=0)}_{H^{s}(\bbR^d\times\bbR^d)}, \qquad \supp_{v}( f(t) ) \subset B_{0}(4R_{0}).
		\end{split}
	\end{equation*}
\end{lemma}

\begin{proof} We proceed in several steps.

\medskip

\noindent \textbf{Step I: $H^{s}$--estimate}. Taking the Fourier transform in $x$ and $v$ and denoting the dual variables by $\xi$ and $\mu$ respectively, we obtain \begin{equation*}
	\begin{split}
		\rd_t \whf(t,\xi,\mu) + \nb_\mu \cdot \xi \whf(t,\xi,\mu) - \int \widehat{u}(t, \xi-\eta) \cdot  i\mu \whf(t,\eta,\mu) \,d\eta = 0 .
	\end{split}
\end{equation*}  With $\brk{\mu}^2 = 1 + |\mu|^2$ and $\brk{\xi}^2 = 1+|\xi|^2$, we have that \begin{equation*}
	\begin{split}
		\frac{1}{2} \frac{d}{dt} \iint (\brk{\mu}^{2s}+\brk{\xi}^{2s})\left| \whf(t,\xi,\mu)\right|^2  \, d\xi d\mu = I + II,
	\end{split}
\end{equation*} where \begin{equation*}
	\begin{split}
		I = -\Re \iint (\brk{\mu}^{2s}+\brk{\xi}^{2s}) \overline{\whf(t,\xi,\mu)} \xi \cdot \nb_\mu \whf (t,\xi,\mu) \, d\xi d\mu  
	\end{split}
\end{equation*} and \begin{equation*}
	\begin{split}
		II = \Re \iint  (\brk{\mu}^{2s}+\brk{\xi}^{2s})\overline{\whf (t,\xi,\mu)}  \int \widehat{u}(t, \xi-\eta) \cdot  i\mu \whf(t,\eta,\mu) \, d\eta d\mu   d\xi. 
	\end{split}
\end{equation*}
After integrating by parts, \begin{equation*}
	\begin{split}
		|I| =\left| \iint \xi \cdot \nb_\mu (\brk{\mu}^{2s}+\brk{\xi}^{2s}) |\whf|^2 \, d\xi d\mu \right| \le C \iint (\brk{\mu}^{2s}+\brk{\xi}^{2s}) | \whf  |^2 \, d\xi d\mu \le C \nrm{f}_{H^s_{x,v}}^2.
	\end{split}
\end{equation*} Next, we observe that from (anti-)symmetry, \begin{equation*}
	\begin{split}
		\Re \iint   \brk{\mu}^{2s} \overline{\whf(t,\xi,\mu)}  \int \widehat{u}(t, \xi-\eta) \cdot  i\mu \whf(t,\eta,\mu) d\eta d\mu   d\xi = 0.
	\end{split}
\end{equation*} Indeed, this follows from changing variables as $(\eta,\xi) \mapsto (\xi,\eta)$ and taking the complex conjugate. Similarly, if we consider the expression \begin{equation*}
	\begin{split}
		II' := \Re \iiint   \brk{\xi}^{ s} \overline{\whf(t,\xi,\mu)}  \widehat{u}(t, \xi-\eta) \cdot  i\mu \brk{\eta}^{s} \whf(t,\eta,\mu) \, d\eta d\mu   d\xi, 
	\end{split}
\end{equation*} we have that $II' = 0$ by symmetry. Therefore, we may rewrite $II$ as \begin{equation*}
	\begin{split}
		II = \Re \iiint   \brk{\xi}^{ s} \overline{\whf(t,\xi,\mu)} (\brk{\xi}^s -\brk{\eta}^s ) \widehat{u}(t, \xi-\eta) \cdot  i\mu \whf(t,\eta,\mu) \, d\eta d\mu   d\xi.
	\end{split}
\end{equation*} 
Using that \begin{equation*}
	\begin{split}
		\left| \brk{\xi}^s -\brk{\eta}^s  \right|\le C\brk{\xi-\eta}(\brk{\xi-\eta}^{s-1} + \brk{\eta}^{s-1}), 
	\end{split}
\end{equation*} we have \begin{equation*}
	\begin{split}
		\left| II \right| \le C(II_1 + II_2 ),
	\end{split}
\end{equation*} where \begin{equation*}
	\begin{split}
		II_1 = \iiint   \brk{\xi}^{ s} |\overline{\whf(t,\xi,\mu)}| \brk{\xi-\eta}^{s}|\widehat{u}(\xi-\eta)| |\mu \whf(t,\eta,\mu)| \, d\eta d\mu   d\xi
	\end{split}
\end{equation*} and \begin{equation*}
	\begin{split}
		II_2 = \iiint   \brk{\xi}^{ s} |\overline{\whf(t,\xi,\mu)}| \brk{\xi-\eta} |\widehat{u}(\xi-\eta)| \brk{\eta}^{s-1}|\mu \whf(t,\eta,\mu)| \, d\eta d\mu   d\xi.
	\end{split}
\end{equation*} Term $II_2$ is easy to handle: first using Cauchy-Schwartz in $\xi,\mu$, Young's convolution inequality in $\eta$, and then Sobolev embedding, \begin{equation*}
	\begin{split}
		\left| II_2 \right|   & \le \nrm{ \brk{\xi}^s \whf }_{L^2_{\xi,\mu}} \left\Vert  \int \brk{\xi-\eta} |\widehat{u}(\xi-\eta)| \brk{\eta}^{s-1}|\mu \whf(t,\eta,\mu)|\, d\eta\right\Vert_{L^2_{\xi,\mu}}\\
		&\le C\nrm{ \brk{\xi}^s \whf }_{L^2_{\xi,\mu}} \nrm{\brk{\xi}\widehat{u}(\xi)}_{L^1_\xi} \nrm{ \brk{\xi}^{s-1}\brk{\mu} \widehat{f} }_{L^2_{\xi,\mu}} \\
		&\le C \nrm{u}_{H^{\frac{d}{2} + 1 + \eps}_x } \nrm{f}_{H^s_{x,v}}^2
	\end{split}
\end{equation*} for any $\eps > 0$. 
We proceed similarly for $II_1$:  \begin{equation*}
	\begin{split}
		\left| II_1 \right|   & \le \nrm{ \brk{\xi}^s \whf(t) }_{L^2_{\xi,\mu}} \left\Vert  \int |\xi-\eta|^{s} |\widehat{u}(\xi-\eta)|  |\mu \whf(t,\eta,\mu)|\, d\eta\right\Vert_{L^2_{\xi,\mu}}\\
		&\le C\nrm{ \brk{\xi}^s \whf(t) }_{L^2_{\xi,\mu}} \nrm{\brk{\xi}^{s}\widehat{u}(\xi)}_{L^2_\xi} \nrm{ \brk{\mu} \widehat{f}(t) }_{L^2_{\mu}L^1_{\xi}} \\
		&\le C \nrm{u}_{H^{s}} \nrm{f}_{H^s_{x,v}} \nrm{f}_{H^{\frac{d}{2}+1+\eps}_{x,v}}.
	\end{split}
\end{equation*}  {Therefore}, we conclude for any $s> \frac{d}{2}+1$ that \begin{equation*}
	\begin{split}
		\frac{d}{dt} \nrm{f}_{H^s_{x,v}}^{2}\le C\nrm{u}_{H^{s}}  \nrm{f }^{2}_{H^s_{x,v}}. 
	\end{split}
\end{equation*}

\medskip

\noindent \textbf{Step II: {Averaging Lemma}}. To apply the averaging lemma, we consider the equation for $\Lmb_x^sf$: \begin{equation*}
	\begin{split}
		\rd_t (\Lmb_x^sf) + v\cdot\nb_x (\Lmb_x^sf) =  - \nb_v \cdot \Lmb_x^s ( u f ). 
	\end{split}
\end{equation*} We shall apply Lemma \ref{lem:avg} with $m=1$, $h = \Lmb^{s}_{x}f$ and $g = \Lmb^{s}_{x}(uf)$. We assume that the support of $f$ in $v$ is compactly contained in the open ball $B(0,R)$ for all $0 \le t \le T$, and take $\psi$ to be a $C^\infty_c$--function which is identically 1 on the support of $\Lmb^{s}_{x}f$ for all $0\le t \le T$. We now estimate 
\begin{equation*}
	\begin{split}
		\nrm{\Lmb_x^sf}_{L^2([0,T];L^2_{x,v})} \le CT^{\frac{1}{2}} \nrm{f}_{L^\infty([0,T]; H^s_{x,v})}
	\end{split}
\end{equation*} and \begin{equation*}
	\begin{split}
		\nrm{\Lmb_x^s ( u f )}_{L^2([0,T];L^2_{x,v})} & \le  C  \nrm{u f }_{L^2([0,T]; L^2_vH^s_{x})} \\
		& \le C \nrm{u}_{L^2([0,T];H^s_x) } \nrm{f}_{L^\infty([0,T]; L^{2}_vH^s_x )} . 
	\end{split}
\end{equation*} We have used that $s>\frac{d}{2}$. We now recall that $u = \pm \Lmb^{-\beta}\nb_{x}\rho$ holds, which gives for each $t$ the estimate \begin{equation*}
	\begin{split}
		\nrm{u(t,\cdot)}_{H^{s} }\le C \nrm{\rho(t,\cdot)}_{H^{s+1-\beta}} . 
	\end{split}
\end{equation*} This gives \begin{equation*}
	\begin{split}
		\nrm{\Lmb_x^s ( u f )}_{L^2([0,T];L^2_{x,v})} & \le  C \nrm{\rho}_{L^{2}_{t}H^{s+1-\beta}}  \nrm{f}_{L^\infty([0,T]; L^{2}_vH^s_x )} . 
	\end{split}
\end{equation*}
We now apply the averaging lemma: since \begin{equation*}
	\begin{split}
		\int \Lmb_{x}^{s}f \psi(v)\, dv = \Lmb_{x}^{s}\rho, 
	\end{split}
\end{equation*} we obtain 
\begin{equation*}
	\begin{split}
		\nrm{\Lmb^s\rho}_{L^2_{t}H^{\frac{1}{4}}_{x}} \le \nrm{\Lmb^s\rho}_{H^{\frac{1}{4}}_{t,x}} \le C(1+R^{\frac{d}{2}})\left(T^{\frac{1}{2}} \nrm{f}_{L^\infty_t L^{2}_{v} H^{s}_{x}} + \nrm{\rho}_{L^{2}_{t}H^{s+1-\beta}}  \nrm{f}_{L^\infty_{t} L^{2}_vH^s_x}   \right). 
	\end{split}
\end{equation*} On the other hand, we can directly estimate \begin{equation*}
	\begin{split}
		\nrm{\rho}_{L^2_{t,x}} \le CR^{\frac{d}{2}} \nrm{f}_{L^2_{t,x,v}} \le CR^{\frac{d}{2}}T^{\frac12} \nrm{f}_{L^\infty_{t}L^2_{x,v}}.
	\end{split}
\end{equation*} Combining the previous estimates and using interpolation \begin{equation*}
	\begin{split}
		\nrm{\rho}_{H^{s+1-\beta}} & \le \nrm{\rho}_{H^{s+\frac14}}^{4(1-\bt)}\nrm{\rho}_{H^{s}}^{1-4(1-\bt)} \\
		&\le   \varepsilon \nrm{\rho}_{H^{s+\frac14}} + C_{\beta} \varepsilon^{-\frac{4(1-\bt)}{1-4(1-\bt)}} \nrm{\rho}_{H^{s}},
	\end{split}
\end{equation*} (here it is used that $\frac{3}{4}<\bt<1$) we deduce that \begin{equation*}
	\begin{split}
		\nrm{\rho}_{L^2_{t}H^{\frac14 + s}_{x}} & \le C(1 + R^{\frac{d}{2}}) \nrm{f}_{L^\infty_{t}L^2_{v}H^{s}_{x}}\left( T^{\frac12} + \nrm{\rho}_{L^{2}_{t}H^{s+1-\beta}}  \right) \\
		&\le C(1 + R^{\frac{d}{2}})\nrm{f}_{L^\infty_{t}L^2_{v}H^{s}_{x}}\left( T^{\frac12} +  \varepsilon \nrm{\rho}_{L^2_tH^{s+\frac14}} + C_{\beta} \varepsilon^{-\frac{4(1-\bt)}{1-4(1-\bt)}} \nrm{\rho}_{L^2_tH^{s}}  \right).
	\end{split}
\end{equation*}
Therefore, we can pick \begin{equation*}
	\begin{split}
		\varepsilon := \frac{1}{10C} \frac{1}{ (1 + R^{\frac{d}{2}})\nrm{f}_{L^\infty_{t}L^2_{v}H^{s}_{x}} }
	\end{split}
\end{equation*} and absorb the $\varepsilon \nrm{\rho}_{L^2_tH^{s+\frac14}}$--term on the right hand side to derive \begin{equation*}
	\begin{split}
		\nrm{\rho}_{L^2_{t}H^{\frac14 + s}_{x}} & \le C(1 + R^{\frac{d}{2}}) \nrm{f}_{L^\infty_{t}L^2_{v}H^{s}_{x}}\left( T^{\frac12} +((1 + R^{\frac{d}{2}})\nrm{f}_{L^\infty_{t}L^2_{v}H^{s}_{x}})^{\frac{4(1-\bt)}{1-4(1-\bt)}} \nrm{\rho}_{L^2_tH^{s}}  \right).
	\end{split}
\end{equation*} Finally, using \begin{equation*}
	\begin{split}
		\nrm{\rho}_{L^2_tH^{s}} \le C T^{\frac12}R^{\frac{d}{2}} \nrm{f}_{L^\infty_t L^2_vH^s_x},
	\end{split}
\end{equation*} we conclude the estimate \begin{equation*}
	\begin{split}
		\nrm{\rho}_{L^2_{t}H^{\frac14 + s}_{x}} & \le C(1 + R^{\frac{d}{2}})T^{\frac12} \nrm{f}_{L^\infty_{t}L^2_{v}H^{s}_{x}} \left(  1  + (1 + R^{\frac{d}{2}}) \nrm{f}_{L^\infty_{t}L^2_{v}H^{s}_{x}}\right)^{\frac{1}{1-4(1-\bt)}} . 
	\end{split}
\end{equation*} 

\medskip

\noindent \textbf{Step III: Closing the a priori estimate}. Returning to the $H^s$ estimate for $f$, we have for $t\le T$ that 
\begin{equation*}
	\begin{split}
		\nrm{f(t) }_{H^s_{x,v}} & \le \nrm{f_0}_{H^s_{x,v}}\exp\left( C\int_0^t \nrm{u(\tau)}_{H^{s}}  d\tau \right) \\
		& \le \nrm{f_0}_{H^s_{x,v}}\exp\left( CT^{\frac12}\nrm{\rho}_{L^2_{t}H^{s+1-\beta}} \right) \\
		& \le  \nrm{f_0}_{H^s_{x,v}}\exp\left( C(1 + R^{\frac{d}{2}})T  \nrm{f}_{L^\infty_{t}L^2_{v}H^{s}_{x}} \left(  1  + (1 + R^{\frac{d}{2}}) \nrm{f}_{L^\infty_{t}L^2_{v}H^{s}_{x}}\right)^{\frac{1}{1-4(1-\bt)}} \right). 
	\end{split}
\end{equation*} We assume that $r(t)>0$ is the smallest satisfying \begin{equation*}
	\begin{split}
		\supp_{v}( f(t,x,\cdot) ) \subset B(0,r(t)), \quad \forall x \in \bbR^{d}
	\end{split}
\end{equation*} and set \begin{equation}\label{eq:R-def}
	\begin{split}
		R(T) = 1 + \sup_{t\in[0,T]} r(t). 
	\end{split}
\end{equation} Furthermore, writing \begin{equation}\label{eq:F-def}
	\begin{split}
		F(T) = \sup_{t\in[0,T]} \nrm{f(t)}_{H^s_{x,v}},
	\end{split}
\end{equation} we have that \begin{equation}\label{eq:F-estimate}
	\begin{split}
		F(T) \le F(0)\exp\left( CT (R^{\frac{d}{2}}(T)F(T))^{1+ \frac{1}{1-4(1-\bt)}}   \right).
	\end{split}
\end{equation} Moreover, \begin{equation*}
	\begin{split}
		\left|\frac{d}{dt} R(t)\right| \le C\nrm{u(t)}_{L^\infty} \le C\nrm{ u(t)}_{H^{\frac{d}{2}+\eps}} \le C \nrm{\rho(t)}_{H^{s}} \le CR^{\frac{d}{2}}(t) \nrm{f(t)}_{H^{s}_{x,v}}.
	\end{split}
\end{equation*} This gives \begin{equation}\label{eq:R-estimate}
	\begin{split}
		R(T) \le R(0)\exp\left( CT R^{\frac{d}{2}-1}(T)F(T) \right). 
	\end{split}
\end{equation}  One can prove that $F(T)$ and $R(T)$ are continuous in $T$, see Subsection \ref{subsec:eandu} below. With a continuity argument in time, we finally conclude the a priori estimate \begin{equation}\label{eq:apriori-compact}
	\begin{split}
		\sup_{t\in[0,T]} \nrm{f(t)}_{H^s_{x,v}} \le 4\nrm{f_0}_{H^s_{x,v}}, \qquad R(T) \le 4R(0) 
	\end{split}
\end{equation} for $T>0$ sufficiently small. This finishes the proof. 
\end{proof}

\subsection{Finite moment case} \label{subsec:apriori-moment}
In this section, we shall take $\sigma>0$ and obtain a priori estimates using Sobolev spaces with finite $v$-moments. 

	\begin{lemma}\label{lem:apriori-moment}
		Assume that $\sigma>0$, $\frac34 < \bt \le 1$, $\kpp \in \{ \pm 1 \}$, $s>\frac{d}{2}+1$, and $N>\frac{d}{4}$. Let $f$ be a sufficiently smooth and fast decaying solution to \eqref{eq:VR} on a time interval $[0,T]$. Then, by taking $T>0$ smaller if necessary, but in a way depending only on $\nrm{f(t=0)}_{H^{s,2N}(\bbR^d\times\bbR^d)}$ and $\sigma$, we have for all $0<t\le T$ \begin{equation*}
			\begin{split}
\sup_{t\in[0,T]}\nrm{f(t, \cdot ) }_{H^{s,2N}_{x,v}} + 2\sigma \int_0^{T} \nrm{ \nb_v ( \brk{v}^{2N} \Lmb^s_{x} f ) }_{L^{2}}^{2} \, dt & \le C(\nrm{f(t=0)}_{ {H^{s,2N}_{x,v}} }, \sigma). 
			\end{split}
		\end{equation*}
	\end{lemma}

\begin{proof}
While the arguments are largely parallel to the compactly supported case, we now need to understand how the moments change in time and incorporate the Fokker--Planck term into the above computations.

\medskip

\noindent \textbf{Step I: $H^{s,2N}_{x,v}$--estimate}. Since the Fourier transform of the operator $\brk{v}^2$ is given by $1 - \lap_\mu$, we have from \begin{equation*}
	\begin{split}
		\rd_t \whf(t,\xi,\mu) + \nb_\mu \cdot \xi \whf(t,\xi,\mu) - \int \widehat{u}(\xi-\eta) \cdot  i\mu \whf(t,\eta,\mu) d\eta = \sigma i\mu \cdot \left( i\mu \whf(t,\xi,\mu) + \frac{1}{i}\rd_{\mu}  \whf(t,\xi,\mu)   \right)
	\end{split}
\end{equation*} that \begin{equation*}
	\begin{split}
		& \frac{1}{2} \frac{d}{dt} \left( \nrm{ \brk{v}^{2N}  {(1 - \lap_{x})^{\frac{s}{2}} } f }_{L^{2}}^{2} + \nrm{  { \brk{v}^{2N} (1 - \lap_{v})^{\frac{s}{2}} } f  }_{L^{2}}^{2}  \right) \\
		& \qquad =   \frac{1}{2} \frac{d}{dt} \iint \left| (1 - \lap_\mu)^{N} \left[ \brk{\xi}^{s} \whf(t,\xi,\mu) \right] \right|^{2} + \left| (1 - \lap_\mu)^{N} \left[ \brk{\mu}^{s} \whf(t,\xi,\mu) \right] \right|^{2}\, d\xi d\mu  \\
		& \qquad = I + II + III,
	\end{split}
\end{equation*} where \begin{equation*}
	\begin{split}
		I & = -\Re \iint  (1 - \lap_\mu)^{ N} \left[ \brk{\xi}^{s} \overline{\whf} \right] (1-\lap_\mu)^{N} \left[ \brk{\xi}^{s} \xi \cdot \nb_\mu \whf \right]  d\xi d\mu  \\
		&\qquad  -\Re \iint  (1 - \lap_\mu)^{ N} \left[ \brk{\mu}^{s} \overline{\whf} \right] (1-\lap_\mu)^{N} \left[ \brk{\mu}^{s} \xi \cdot \nb_\mu \whf \right]  d\xi d\mu ,
	\end{split}
\end{equation*}  \begin{equation*}
	\begin{split}
		II &=   \Re \iint  (1 - \lap_\mu)^{ N} \left[   \brk{\xi}^{ s} \overline{\whf(t,\xi,\mu)}  \right] (1-\lap_\mu)^{N} \left[ \brk{\xi}^{s} \int \widehat{u}(\xi-\eta) \cdot  i\mu \whf(t,\eta,\mu) d\eta\right] d\mu   d\xi \\
		&\qquad +  \Re \iint  (1 - \lap_\mu)^{ N} \left[   \brk{\mu}^{ s} \overline{\whf(t,\xi,\mu)}  \right] (1-\lap_\mu)^{N} \left[ \brk{\mu}^{s} \int \widehat{u}(\xi-\eta) \cdot  i\mu \whf(t,\eta,\mu) d\eta\right] d\mu   d\xi
		\\
		& =: II_\xi + II_\mu,
	\end{split}
\end{equation*}  and \begin{equation*}
	\begin{split}
		III & = \sigma  \Re \iint  (1 - \lap_\mu)^{ N} \left[   \brk{\xi}^{ s} \overline{\whf}  \right] (1-\lap_\mu)^{N} \brk{\xi}^{s}  \left[ -|\mu|^{2} \whf + \mu\cdot\nb_\mu \whf  \right] d\mu   d\xi \\
		&\qquad +  \sigma \Re \iint  (1 - \lap_\mu)^{ N} \left[   \brk{\mu}^{ s} \overline{\whf}  \right] (1-\lap_\mu)^{N} \brk{\mu}^{s} \left[  -|\mu|^{2} \whf + \mu\cdot\nb_\mu \whf \right] d\mu   d\xi
		\\
		& =: III_\xi + III_\mu. 
	\end{split}
\end{equation*}
The term $I$ can be handled similarly as before; since $(1-\lap_\mu)^N$ commutes with $\xi\cdot\nb_\mu$, we obtain  \begin{equation*}
	\begin{split}
		|I| \le C  \nrm{f}_{H^{s,2N}_{x,v}}^2.
	\end{split}
\end{equation*} Next,  {to treat $II_\xi$, we first note that} \begin{equation*}
	\begin{split}
		(1-\lap_\mu)^N [\mu \whf] = \mu (1-\lap_\mu)^{N}[\whf] - 2N (1-\lap_\mu)^{N-1} [ \nb_\mu\whf  ]
	\end{split}
\end{equation*}  {holds, which gives rise to two terms. The first term can be treated similarly as the corresponding term from the compact support case (replacing $\brk{\xi}^s$ by $\brk{\xi}^{s}- \brk{\eta}^s$), with bound \begin{equation*}
		\begin{split}
			C \left( \nrm{ \brk{\xi}\widehat{u}(\xi) }_{L^1}  \nrm{ (1 - \lap_\mu)^{N} \brk{\xi}^s \whf}_{L^{2}}^{2} + \nrm{ \brk{\xi}^{s}\widehat{u}(\xi) }_{L^{2}}  \nrm{ (1 - \lap_\mu)^{N} \whf}_{L^{1}}^{2} \right).
		\end{split}
	\end{equation*} On the other hand, the other term can be bounded directly, with bound  \begin{equation*}
		\begin{split}
			C \nrm{ \widehat{u} (\xi)}_{L^1} \nrm{ (1 - \lap_\mu)^{N} \brk{\xi}^s \whf}_{L^{2}}   \nrm{ (1 - \lap_\mu)^{N-1}  \nb_\mu \brk{\xi}^s \whf}_{L^{2}}. 
		\end{split}
	\end{equation*} Here, we note that \begin{equation*}
		\begin{split}
			\nrm{ (1 - \lap_\mu)^{N-1}  \nb_\mu \brk{\xi}^s \whf}_{L^{2}}\le C \nrm{ (1 - \lap_\mu)^{N} \brk{\xi}^s \whf}_{L^{2}} 
		\end{split}
	\end{equation*} holds, which can be seen by taking the inverse Fourier transform in the $\mu$-variable. Together, we obtain that} \begin{equation*}
	\begin{split}
		\left|II_\xi\right| \le C\nrm{u}_{H^{s}} \nrm{ f }_{H^{s,2N}_{x,v}}^2. 
	\end{split}
\end{equation*} Moving on to the estimate of $II_\mu$, observing the cancellation 
\begin{equation*}
	\begin{split}
		\Re \iint  (1-\lap_\mu)^{N} \left[\brk{\mu}^{s} \overline{\whf}(t,\xi,\mu) \right]  \int \widehat{u}(\xi-\eta) \cdot  i\mu (1-\lap_\mu)^{N}\left[  {\brk{\mu}^{s} } \whf(t,\eta,\mu) \right]  d\eta   d\mu   d\xi = 0,
	\end{split}
\end{equation*}  we have that $II_\mu$ can be rewritten as  \begin{equation*}
	\begin{split}
		II_\mu =  \Re \iint  (1 - \lap_\mu)^{ N} \left[   \brk{\mu}^{ s} \overline{\whf} (t,\xi,\mu) \right]   \int \widehat{u}(\xi-\eta) \cdot  i [(1-\lap_\mu)^{N}, \mu] \brk{\mu}^{s}  \whf(t,\eta,\mu) \,d\eta d\mu   d\xi,
	\end{split}
\end{equation*} where $[(1-\lap_\mu)^{N}, \mu ]$ is the commutator between $(1-\lap_\mu)^{N}$ and $\mu$.  {It is not difficult to obtain the estimate \begin{equation}\label{eq:comm}
		\begin{split}
			\nrm{    [(1-\lap_\mu)^{N}, \mu] \brk{\mu}^{s}  \whf }_{L^2_{\xi,\mu}} \le C \nrm{ f }_{H^{s,2N}_{x,v}},
		\end{split}
	\end{equation} which implies} \begin{equation*}
	\begin{split}
		\left| II_\mu\right| \le C \nrm{u}_{H^{s}} \nrm{ f }_{H^{s,2N}_{x,v}}^2. 
	\end{split}
\end{equation*}  {To prove a general commutator estimate from which \eqref{eq:comm} follows as a special case, we first note an inductive formula: for $k\ge1$ and a smooth function $G$ of $\mu$, we have \begin{equation}\label{eq:comm-i}
		\begin{split}
			[(1-\lap_\mu)^{k+1}, G(\mu) ]  = (1-\lap_{\mu}) [(1-\lap_\mu)^{k}, G(\mu) ] + [(1-\lap_\mu) , G(\mu) ] (1-\lap_{\mu})^{k}, 
		\end{split}
	\end{equation} which follows from a direct computation. Next, when $N=1$ and $G(\mu)=\mu$, we have \begin{equation*}
		\begin{split}
			[(1-\lap_\mu) , \mu  ] =  - \nb_\mu. 
		\end{split}
	\end{equation*} Using the above inductive formula, it follows that \begin{equation}\label{eq:comm1}
		\begin{split}
			[(1-\lap_\mu)^{N} , \mu  ] = -N  (1-\lap_\mu)^{N-1}\nb_\mu ,
		\end{split}
	\end{equation} from which \eqref{eq:comm} follows. 
}

We now estimate $III$; in the case of $III_\xi$, introducing $H:= \brk{\xi}^s \whf$ for simplicity, we have that \begin{equation*}
	\begin{split}
		III_\xi  & =  \sigma  \Re \iint  (1 - \lap_\mu)^{ N} H (1-\lap_\mu)^{N} \left[ -|\mu|^{2} H + \mu\cdot\nb_\mu H \right] d\mu   d\xi \\
		& = -\sigma \iint \left| |\mu| (1-\lap_\mu)^{N} H \right|^2 - \frac{\mu}{2}\cdot \nb_\mu \left| (1-\lap_\mu)^{N} H  \right|^{2}   d\mu d\xi \\
		& \qquad + \sigma \Re \iint  (1 - \lap_\mu)^{ N} H \, \left[ (1-\lap_\mu)^{N} ,  ( -|\mu|^{2} + \mu\cdot\nb_\mu)  \right] H  d\mu   d\xi,
	\end{split}
\end{equation*} where $ \left[ (1-\lap_\mu)^{N} ,  ( -|\mu|^{2} + \mu\cdot\nb_\mu)  \right]$ is the commutator between the operators $(1-\lap_\mu)^{N}$ and $( -|\mu|^{2} + \mu\cdot\nb_\mu)$.  {To begin with, note that \begin{equation}\label{eq:comm2}
		\begin{split}
			\left[ (1-\lap_\mu)^{N} , \mu\cdot\nb_\mu  \right] =\left[ (1-\lap_\mu)^{N} , \mu \right]  \cdot\nb_\mu  =  -N  (1-\lap_\mu)^{N-1} \lap_{\mu} ,
		\end{split}
	\end{equation} using \eqref{eq:comm1}. Next, we claim that \begin{equation*}
		\begin{split}
			\left[ (1-\lap_\mu)^{N} , |\mu|^{2} \right] = C_{N} \mu\cdot \nb_\mu (1-\lap_\mu)^{N-1} + \sum_{j=0}^{N-1} C_{j,N} \lap_\mu^{j}  
		\end{split}
	\end{equation*}  holds for some combinatorial constants $C_{j,N}$ and $C_{N}$. This assertion can be easily proved with an induction in $N$, using \eqref{eq:comm-i} and \eqref{eq:comm2}.
} Therefore, we can bound \begin{equation}\label{eq:III}
	\begin{split}
		\left| III_\xi + \sigma \iint \left|\brk{\mu}(1-\lap_\mu)^{N} H \right|^2 + d\left| (1-\lap_\mu)^{N} H  \right|^{2}  d\mu d\xi  \right| \le C\sigma \nrm{ (1-\lap_\mu)^{N} H }_{L^{2}_{\xi,\mu}}^{2}. 
	\end{split}
\end{equation}  {In the above, we have simply used that $\nrm{ \lap_\mu^{j} H}_{L^2} \le \nrm{ (1-\lap_\mu)^N H}_{L^2}$ for any $j\le N$, which can be seen by taking the inverse Fourier transform.} The estimate of $III_\mu$ is similar; we can obtain \eqref{eq:III} with $III_\xi$ and $H$ replaced with $III_\mu$ and $\brk{\mu}^s \whf$, respectively.

Collecting the estimates for $I$, $II$, and $III$, we conclude for $s> \frac{d}{2}+1$ that \begin{equation}\label{eq:Hs-moment}
	\begin{split}
		\frac12 \frac{d}{dt} \nrm{f}_{H^{s,2N}_{x,v}}^{2} +  \sigma \left( \nrm{ \nb_v ( \brk{v}^{2N} \Lmb^s_{x} f ) }_{L^{2}}^{2} + \nrm{ \nb_v ( \brk{v}^{2N} \Lmb^s_{v} f ) }_{L^{2}}^{2}  \right)      \le C(\nrm{u}_{H^{s}} +\sigma  )\nrm{f }_{H^{s,2N}_{x,v}}^{2}. 
	\end{split}
\end{equation}

\medskip

\noindent \textbf{Step II: Averaging lemma via moments}. We now fix a smooth and radial bump function $\psi(v)\ge 0$ supported on $\{ |v| \le 2 \}$ and equals 1 on $\{ |v| \le 1 \}$. Then, we set $\phi_j(v) := \psi(2^{-j}v) - \psi(2^{-j+1}v)$ and $\phi_0 = \psi$ so that $1 = \sum_{j\ge 0} \phi_j$. From \begin{equation*}
	\begin{split}
		\rd_t f + v\cdot\nb_x f = - \nb_v\cdot(uf) + \sigma \nb_v \cdot ( \nb_v f + vf ), 
	\end{split}
\end{equation*} we have \begin{equation*}
	\begin{split}
		\rd_t( \phi_j(v) \Lmb^s_x f) + v\cdot\nb_x (\phi_j(v) \Lmb^s_x f) &= - \nb_v \cdot (\phi_j(v) \Lmb^s_x (uf)) + \nb_v \phi_j (v) \cdot \Lmb_x^s(uf) \\
		&\qquad + \sigma \nb_v \cdot ( \phi_j(v)\nb_v  (\Lmb^s_xf) + v \phi_j(v) \Lmb^s_xf )  \\
		&\qquad + \sigma \nb_v\phi_j(v) \cdot ( \nb_v (\Lmb_x^sf) + v\Lmb_x^sf ).
	\end{split}
\end{equation*} Using that  \begin{equation*}
	\begin{split}
		\Lmb^s\rho(x) = \sum_{j\ge 0} \int \Lmb^s_x f(x,v) \phi_j(v) \, dv ,
	\end{split}
\end{equation*} we shall apply the Averaging Lemma for each $j$ to $\phi_j(v) \Lmb^s_x f$: \begin{equation*}
	\begin{split}
		&\left\Vert  \int \Lmb^s_x f(x,v) \phi_j(v) \, dv  \right\Vert_{L^2_{t}H^{\frac14}_{x}} \\
		& \le C2^{\frac{(j+1)d}{2}}\left( \nrm{ \phi_j(v) \Lmb^s_x f }_{L^2_{t,x,v}} + \nrm{ \phi_j(v) \Lmb^s_x (uf)}_{L^2_{t,x,v}} + \nrm{  \nb \phi_j (v) \cdot \Lmb_x^s(uf) }_{L^2_{t,x,v}} \right.  \\
		&\qquad + \left. \sigma \nrm{ \phi_j(v)\nb_v  (\Lmb^s_xf) + v \phi_j(v) \Lmb^s_xf  }_{L^2_{t,x,v}} + \sigma \nrm{\nb_v\phi_j(v) \cdot ( \nb_v (\Lmb_x^sf) + v\Lmb_x^sf ) }_{L^{2}_{t,x,v}} \right) \\
		& \le C2^{\frac{(j+1)d}{2}}( 1 + \sigma 2^{j}) 2^{-2N(j-1)} \left( ( T^{\frac12}  + \nrm{\rho}_{L^2_t H^{s+1-\beta}}) \nrm{f}_{L^\infty_t H^{s,2N}_{x,v}} + \sigma\nrm{ \nb_v ( \brk{v}^{2N} \Lmb_x^s f )  }_{L^2_{t,x,v}} \right).
	\end{split}
\end{equation*} Therefore, for $N> \frac{d}{4}$ ( {this implies that $N \ge \frac{d}{4}+1$ since $N$ was assumed to be an integer}), we obtain from triangle inequality that \begin{equation*} 
	\begin{split}
		\nrm{ \Lmb^s\rho}_{L^2_t H^{\frac14}_x} &\le \sum_{j\ge 0} \left\Vert  \int \Lmb^s_x f(x,v) \phi_j(v) \, dv  \right\Vert_{L^2_{t}H^{\frac14}_{x}} \\ 
		& \le C \left( ( T^{\frac12}  + \nrm{\rho}_{L^2_t H^{s+1-\beta}}) \nrm{f}_{L^\infty_t H^{s,2N}_{x,v}} + \sigma\nrm{ \nb_v ( \brk{v}^{2N} \Lmb_x^s f )  }_{L^2_{t,x,v}} \right).
	\end{split}
\end{equation*} Interpolating similarly as in the compactly supported case, \begin{equation*}
	\begin{split}
		\nrm{\rho}_{L^2_{t} \dot{H}^{\frac14 + s}_{x}} & \le C  \nrm{f}_{L^\infty_{t} H^{s,2N}_{x,v}}\left( T^{\frac12} +(1 + \nrm{f}_{L^\infty_{t} H^{s,2N}_{x,v}})^{\frac{4(1-\bt)}{1-4(1-\bt)}} \nrm{\rho}_{L^2_tH^{s}}  \right) + C\sigma\nrm{ \nb_v ( \brk{v}^{2N} \Lmb_x^s f )  }_{L^2_{t,x,v}}  .
	\end{split}
\end{equation*} On the other hand, from the Cauchy--Schwartz inequality, we obtain that \begin{equation*}
	\begin{split}
		\nrm{\rho}_{L^2_tH^{s}} \le C T^{\frac12} \nrm{f}_{L^\infty_t  H^{s,2N}_{x,v}},
	\end{split}
\end{equation*} for $N>\frac{d}{4}$. Therefore, we conclude the estimate \begin{equation*}
	\begin{split}
		\nrm{\rho}_{L^2_{t}H^{\frac14 + s}_{x}} & \le C T^{\frac12} \nrm{f}_{L^\infty_{t} H^{s,2N}_{x,v}} \left(  1  +  \nrm{f}_{L^\infty_{t} H^{s,2N}_{x,v}}\right)^{\frac{1}{1-4(1-\bt)}} + C\sigma\nrm{ \nb_v ( \brk{v}^{2N} \Lmb_x^s f )  }_{L^2_{t,x,v}}  . 
	\end{split}
\end{equation*} 

\medskip

\noindent \textbf{Step III: Closing the a priori estimate}. Returning to \eqref{eq:Hs-moment} and integrating in time from $t = 0$ to $t = T$, we obtain that \begin{equation*}
	\begin{split}
		\nrm{f(T,\cdot)}_{H^{s,2N}_{x,v}}^{2} +  2\sigma \int_0^{T} \nrm{ \nb_v ( \brk{v}^{2N} \Lmb^s_{x} f ) }_{L^{2}}^{2} \, dt  \le \nrm{f_0}_{H^{s,2N}_{x,v}}^{2} + \int_0^T C(\nrm{u}_{H^{s}} +\sigma  )\nrm{f }_{H^{s,2N}_{x,v}}^{2} \, dt . 
	\end{split}
\end{equation*}
In the above inequality, the same estimate holds for $\nrm{f(T,\cdot)}_{H^{s,2N}_{x,v}}^{2}$ replaced with any $\nrm{f(\tau,\cdot)}_{H^{s,2N}_{x,v}}^{2}$ as long as $0\le \tau<T$. To ease notation, let us introduce \begin{equation*}
	\begin{split}
		X := \sup_{t\in[0,T]} \nrm{f(t,\cdot)}_{H^{s,2N}_{x,v}}^{2}, \qquad A(t) :=  \nrm{ \nb_v ( \brk{v}^{2N} \Lmb^s_{x} f(t, \cdot) ) }_{L^{2}_{x,v}} .
	\end{split}
\end{equation*} Then, we simply have \begin{equation*}
	\begin{split}
		X + 2\sigma\int_0^T A^{2}(t) \, dt \le C\left(   1 +   \sigma TX + X  \int_0^T \nrm{u}_{H^{s}} \,dt   \right).
	\end{split}
\end{equation*} Note that at this point we can absorb the $\sigma TX$ term on the right hand side to the left hand side by taking $T \le 1/(2\sigma)$. Then, we recall the averaging lemma: \begin{equation*}
	\begin{split}
		\int_0^T \nrm{u}_{H^{s}} \,dt \le CT^{\frac12}\left(T^{\frac12} X \left(  1  +  X \right)^{\frac{1}{1-4(1-\bt)}} +  \sigma \nrm{A}_{L^2_t([0,T])} \right).
	\end{split}
\end{equation*} This gives the estimate \begin{equation}\label{eq:prelim}
	\begin{split}
		X + 4\sigma\int_0^T A^{2}(t) \, dt \le C(1+X)T^{\frac12}\left(T^{\frac12} X \left(  1  +  X \right)^{\frac{1}{1-4(1-\bt)}} +  \sigma \nrm{A}_{L^2_t([0,T])} \right).
	\end{split}
\end{equation} On the right hand side, the term involving $A$ can be handled as follows: \begin{equation*}
	\begin{split}
		C(1+X) T^{\frac12}  \sigma \nrm{A}_{L^2_t([0,T])}  \le C_\sigma(1+X)^2T + 2  \sigma \nrm{A}_{L^2_t([0,T])}^{2}.
	\end{split}
\end{equation*} Then, the last term on the right hand side can be absorbed into the left hand side of \eqref{eq:prelim}. This gives \begin{equation*}
	\begin{split}
		X + 2\sigma\int_0^T A^{2}(t) \, dt \le C(1+X)T^{\frac12}\left(T^{\frac12} X \left(  1  +  X \right)^{\frac{1}{1-4(1-\bt)}} +  C_\sigma(1+X)^2T \right) . 
	\end{split}
\end{equation*} The right hand side is small when $T\ll1$. Therefore, with a standard continuity argument in time, we conclude that \begin{equation*}
	\begin{split}
		\sup_{t\in[0,T]}\nrm{f(t, \cdot ) }_{H^{s,2N}_{x,v}} + 2\sigma \int_0^{T} \nrm{ \nb_v ( \brk{v}^{2N} \Lmb^s_{x} f ) }_{L^{2}}^{2} \, dt & \le C(\nrm{f_0}_{ {H^{s,2N}_{x,v}} }, \sigma) 
	\end{split}
\end{equation*} for some $T = T(\nrm{f_0}_{ {H^{s,2N}_{x,v}} }, \sigma )>0$.  This finishes the proof of the desired a priori estimate. \end{proof}

\subsection{Existence and uniqueness} \label{subsec:eandu}
In the above, we have obtained a priori estimates both in the compactly supported and the finite moment cases. Given these estimates, to conclude Theorem \ref{thm:lwp} it only remains to prove existence and uniqueness of solutions, which is rather standard. 

\begin{proof}[Proof of Theorem \ref{thm:lwp}]
We shall sketch a proof in the case $\sigma = 0$, assuming compact support in $v$. The case $\sigma>0$ does not require any significant change in the arguments.  Let us fix some $\frac34 < \beta \le 1$ and initial data $f_0 \in H^{s}_{x,v}(\bbR^d\times\bbR^d)$ compactly supported in $v$. To obtain existence of an $L^\infty_tH^s$--solution to \eqref{eq:VR} with initial data $f_0$, we may consider the regularized system
\begin{equation} \label{eq:VR-reg}
	\left\{
	\begin{aligned}
		&\rd_t f^\eps + v\cdot\nb_x f^\eps  \pm \Lmb^{-\beta}( \varphi_\eps * \nb_x \rho^\eps ) \cdot \nb_v f^\eps  = 0, \\
		& \rho^\eps  = \intr f^\eps  \, dv, \\
		& f^\eps(t=0) = f_0. 
	\end{aligned}
	\right.
\end{equation} Here $\varphi_\eps(x) := \eps^{-d}\varphi(\eps^{-1}x)$ is the standard mollifier scaled to $\eps>0$. Then, for each $\eps>0$, it is not difficult to establish existence (e.g. following the same argument as in the Vlasov--Poisson system \cite[Chapter 4]{Gla96}) of a solution $f^\eps$ belonging to $L^{\infty}([0,T_\eps); H^{s}_{x,v})$ and compactly supported in $v$. While in principle $T_\eps>0$ could decrease as $\eps\to0$, the time interval of existence can be extended as long as the support radius in $v$ and the quantity $\nrm{f(t,\cdot)}_{H^{s}_{x,v}}$ remains bounded. By defining $F^\eps$ and $R^\eps$ by \eqref{eq:F-def} and \eqref{eq:R-def} with $f$ replaced with $f^\eps$, one can repeat the proof of a priori estimates for \eqref{eq:VR-reg} to derive \eqref{eq:F-estimate} and \eqref{eq:R-estimate} with $(F,R)$ replaced with $(F^\eps, R^\eps)$. Furthermore, it is clear that $R^\eps$ is continuous in time. Therefore, applying a continuity argument in time is justified, and we conclude \eqref{eq:apriori-compact} for  $(F^\eps, R^\eps)$, which in particular guarantees uniform-in-$\eps$ existence time interval $[0,T]$ with $T>0$ depending only on $f_0$. The sequence $\{ f^\eps \}_{\eps>0}$ has a weakly convergent subsequence in $L^\infty([0,T];H^{s}_{x,v})$, and we denote the limit by $f$. Then, \begin{itemize}
\item $f$ is compactly supported in $v$ and belongs to $L^\infty([0,T];H^{s}_{x,v})$.
\item $f(t=0)=f_0$. 
\item $f$ is a solution to \eqref{eq:VR}. 
\end{itemize}
This gives existence. To prove uniqueness, assume that there are two compactly supported solutions $f, \bar{f}$ with the same initial data $f_0$. We may assume that the solutions belong to $L^\infty([0,T];H^s)$ with $T>0$ and supported in $v \in B(0,R)$ for some $R\ge 1$ for all $t \in [0,T]$. We denote the corresponding velocities by $u = \pm \Lmb^{-\bt} \nb\rho $ and $\bar{u} = \pm \Lmb^{-\bt}\nb\bar{\rho}$ where $\rho = \int f dv$ and $\bar{\rho} = \int \bar{f} dv$. Then, the equation for $f- \bar{f}$ reads \begin{equation*}
	\begin{split}
		\frac{d}{dt} (f - \bar{f}) + v \cdot \nb_x (f - \bar{f}) + u \cdot \nb_v (f - \bar{f}) + (u - \bar{u}) \cdot \nb \bar{f} = 0. 
	\end{split}
\end{equation*} We estimate for $t \in [0,T]$ \begin{equation}\label{eq:uniq}
\begin{split}
	\frac{d}{dt} \nrm{f - \bar{f}}_{L^{2}_{x,v}}^2 \le C_R\nrm{ \nb \bar{f}}_{L^\infty}   \nrm{u - \bar{u}}_{L^{2}_{x}}\nrm{f - \bar{f}}_{L^{2}_{x,v}},
\end{split}
\end{equation} with $C_R>0$ depending on $R$. Then, for some small $0<\tau\le T$ to be determined, the averaging lemma applied to $f - \bar{f}$ gives \begin{equation*}
\begin{split}
	\nrm{\rho-\bar{\rho}}_{L^2([0,\tau];H^{\frac14}_x)} \le C_{R,f,\bar{f}} \tau^{\frac12} \nrm{f - \bar{f}}_{L^\infty([0,\tau];L^{2}_{x,v})}. 
\end{split}
\end{equation*} Here and in the following, $C_{R,f,\bar{f}}$ is a constant which depends only on $R$ and $\nrm{f}_{L^\infty([0,T];H^s)}$, $\nrm{\bar{f}}_{L^\infty([0,T];H^s)}$. This gives \begin{equation*}
\begin{split}
	\nrm{u-\bar{u}}_{L^2([0,\tau];L^{2}_x)} \le C_{R,f,\bar{f}} \tau^{\frac12} \nrm{f-\bar{f}}_{L^\infty([0,\tau];L^{2}_x)}. 
\end{split}
\end{equation*}  Returning to \eqref{eq:uniq} and integrating in time, we have from $f-\bar{f} = 0$ at $t = 0$ that \begin{equation*}
\begin{split}
	\nrm{f-\bar{f}}_{L^2_{x,v}}(\tau) & \le C_{R,f,\bar{f}}\int_0^\tau \nrm{u-\bar{u}}_{L^2} (s) \,ds \\
	& \le C_{R,f,\bar{f}}\tau \nrm{f - \bar{f}}_{L^\infty([0,\tau];L^{2}_x)}.
\end{split}
\end{equation*} We could have taken $\tau>0$ small in a way that $C_{R,f,\bar{f}}\tau<\frac12$. This forces that \begin{equation*}
\begin{split}
	\nrm{f - \bar{f}}_{L^\infty([0,\tau];L^{2}_x)} = 0. 
\end{split}
\end{equation*} Repeating the same argument starting from $t =\tau$ gives $f = \bar{f}$ on $[0,T]$. This finishes the proof of uniqueness. \end{proof}

\section{Singularity formation}\label{sec_blow}

\subsection{Proof of Theorem \ref{thm:blow}}
In this section, we give the details on the proof for Theorem \ref{thm:blow}. For this, we first start with the total energy estimate for the system \eqref{eq:VR}. 

\begin{lemma}\label{lem_energy} Let $f$ be a strong solution to the system \eqref{eq:VR} satisfying $f \in X(T)$. Then we have for $t<T$
\[
E(t) + \sigma\int_0^t \intrr \frac1f |\nabla_v f + vf|^2\,dxdvds = E(0).
\]
\end{lemma}
\begin{proof} Straightforward computations give
\begin{align*}
\frac12\frac{d}{dt}\left(\intrr |v|^2 f \,dxdv\right) 
&= \intrr( \nabla K\star\rho)\cdot v f \,dxdv  - \sigma\intrr v \cdot (\nabla_v f + vf) \,dxdv.
\end{align*}
On the other hand, we get
\[
\frac{d}{dt}\left(\frac{1}{2}\intr \rho K\star\rho \,dx\right) = \intr \partial_t\rho K\star\rho \,dx = \intr \rho u \cdot \nabla K\star\rho \,dx = \intrr( \nabla K\star\rho)\cdot v f \,dxdv.
\]
This yields
\[
\frac{d}{dt}\left(\frac12\intrr |v|^2 f \,dxdv - \frac{1}{2}\intr \rho K\star\rho \,dx\right)= - \sigma\intrr v \cdot (\nabla_v f + vf) \,dxdv.
\]
We then combine the previous estimates with the following entropy estimate
\begin{align*}
\frac{d}{dt}\left(\intrr  f \log f \,dxdv\right)
&= \intrr (\partial_t f) \log f\,dxdv\\
&= - \intrr (   \nabla K\star\rho)\cdot \nabla_v f \,dxdv -\sigma \intrr \frac{\nabla_v f}{f} \cdot (\nabla_v f + vf)\,dxdv
\end{align*}
to conclude the desired result.
\end{proof}

\begin{remark}\label{rmk_s0} In the absence of diffusion, i.e. $\sigma = 0$, we easily find from Lemma \ref{lem_energy} that 
\[
\frac{d}{dt}\left(\frac12\intrr |v|^2 f \,dxdv - \frac{1}{2}\intr \rho K\star\rho \,dx\right) = 0.
\]
Thus if we define $\tilde E = \tilde E(t)$ as
\[
\tilde E = \frac12\intrr |v|^2 f \,dxdv - \frac{1}{2}\intr \rho K\star\rho \,dx,
\]
then $\tilde E(t) = \tilde E(0)$ for all $t\geq 0$.
\end{remark}

Then in the lemma below, we show the estimate on the second-order time-derivative of $I$.
\begin{lemma}\label{lem_mom} Let $f$ be a solution to the system \eqref{eq:VR} satisfying $f \in X(T)$. Then we have
\[
I''(t) =  \intrr |v|^2 f\,dxdv -\frac12\sum_{i=1}^Nc_i \alpha_i\intr \rho K_i \star \rho\,dx - \sigma I'(t).
\]
\end{lemma}
\begin{proof} Straightforward computations gives
\begin{align*}
\frac12\frac{d^2}{dt^2}\intrr |x|^2f\,dxdv &= \frac{d}{dt}\intrr (x \cdot v) f \,dxdv \cr
&= - \intrr (x\cdot v) \lt(\nabla_x \cdot (vf) + \nabla_v \cdot ((\nabla K\star\rho) f) - \sigma\nabla_v \cdot(\nabla_vf + vf) \rt) dxdv\cr
&= \intrr |v|^2 f\,dxdv + \intr  x \cdot (\nabla K\star\rho)\rho\,dx - \sigma \intrr (x \cdot v) f \,dxdv \cr
&= \intrr |v|^2 f\,dxdv -\frac12\sum_{i=1}^Nc_i \alpha_i\intr \rho K_i \star \rho\,dx - \sigma I'(t),
\end{align*}
where we used
\[
\intor (x\cdot v) \Delta _v f\,dxdv = 0
\]
and
\begin{align*}
\int \rho x \cdot \nabla K\star\rho\,dx &= -\sum_{i=1}^Nc_i \alpha_i \intrr \rho(x) x \cdot \frac{x-y}{|x-y|^{\alpha_i+2}} \rho(y)\,dxdy\cr
&=\sum_{i=1}^Nc_i \alpha_i \intrr \rho(y) y \cdot \frac{x-y}{|x-y|^{\alpha_i+2}} \rho(x)\,dxdy\cr
&=-\frac12\sum_{i=1}^Nc_i \alpha_i \intrr \rho(x) \frac{1}{|x-y|^{\alpha_i}} \rho(y)\,dxdy\cr
&=-\frac12\sum_{i=1}^Nc_i \alpha_i\intr \rho K_i \star \rho\,dx.
\end{align*}
This implies
\[
I''(t) =  \intrr |v|^2 f\,dxdv -\frac12\sum_{i=1}^Nc_i \alpha_i\intr \rho K_i \star \rho\,dx - \sigma I'(t) \qedhere 
\]
\end{proof}
Using the above estimates on the momentum of inertia, we obtain Theorem \ref{thm:blow} in the case $\sigma = 0$.
\begin{proof}[Proof of Theorem \ref{thm:blow} in the case $\sigma = 0$] It follows from Lemma \ref{lem_mom} that 
\[
\frac12\frac{d^2}{dt^2}\intrr |x|^2f\,dxdv  = \intrr |v|^2 f\,dxdv - \frac12\sum_{i=1}^Nc_i \alpha_i\intr \rho K_i \star \rho\,dx.
\]
This together with Remark \ref{rmk_s0} yields
\bq\label{est_mom}
\frac12\frac{d^2}{dt^2}\intrr |x|^2f\,dxdv = 2\tilde E(0) +  \sum_{i=1}^Nc_i \lt(1 - \frac{\alpha_i}2\rt)\intr \rho K_i\star\rho\,dx.
\eq
Since $\max_{i=1,\dots,N}\alpha_i > 2$, without loss of generality, we may assume $\alpha_1 = \max_{i=1,\dots,N}\alpha_i  > 2$. Note that if $\alpha_i > \alpha_j$, then for any $\e >0$,
\begin{align*}
\intr \rho K_j\star\rho\,dx &= \lt(\int_{|x-y|\leq \e} + \int_{|x-y|\geq  \e}\rt)  \rho(x) \frac{1}{|x-y|^{\alpha_j}} \rho(y)\,dxdy\cr
&\leq \e^{\alpha_i - \alpha_j} \int_{|x-y|\leq \e} \rho(x) \frac{1}{|x-y|^{\alpha_i}} \rho(y)\,dxdy + \frac1{\e^{\alpha_j}}\|\rho\|_{L^1}^2\cr
&\leq \frac1{\e^{\alpha_j}}\lt( \e^{\alpha_i} \intr \rho K_i\star\rho\,dx + 1\rt).
\end{align*}
Using the above observation, we estimate
\begin{align*}
&\sum_{i=1}^N c_i \lt(1 - \frac{\alpha_i}2\rt)\intr \rho K_i\star\rho\,dx \cr
&\quad = \lt(\sum_{\alpha_i > 2} + \sum_{\alpha_i < 2}\rt) c_i \lt(1 - \frac{\alpha_i}2\rt)\intr \rho K_i\star\rho\,dx\cr
&\quad \leq c_1 \lt(1 - \frac{\alpha_1}2\rt)\intr \rho K_1\star\rho\,dx + \sum_{\alpha_i < 2}c_i \lt(1 - \frac{\alpha_i}2\rt)\intr \rho K_i\star\rho\,dx\cr
&\quad \leq c_1 \lt(1 - \frac{\alpha_1}2\rt)\intr \rho K_1\star\rho\,dx + \sum_{\alpha_i < 2}c_i \lt(1 - \frac{\alpha_i}2\rt)\frac1{\e^{\alpha_i}}\lt(\e^{\alpha_1}\intr \rho K_1\star\rho\,dx + 1\rt)\cr
&\quad \leq c_1 \lt(1 - \frac{\alpha_1}2\rt)\intr \rho K_1\star\rho\,dx + \sum_{\alpha_i < 2}c_i \lt(1 - \frac{\alpha_i}2\rt)\frac1{\e^2}\lt(\e^{\alpha_1}\intr \rho K_1\star\rho\,dx + 1\rt).
\end{align*}
We now choose $\e>0$ such that 
\[
\e = \lt(c_1 \lt(\frac{\alpha_1}2 - 1\rt) \lt(\sum_{\alpha_i < 2}c_i \lt(1 - \frac{\alpha_i}2\rt)\rt)^{-1}\rt)^{\frac1{\alpha_1-2}},
\]
then this gives
\[
\sum_{i=1}^N c_i \lt(1 - \frac{\alpha_i}2\rt)\intr \rho K_i\star\rho\,dx \leq \lt( \sum_{\alpha_i < 2}c_i \lt(1 - \frac{\alpha_i}2\rt)\rt)^{1 - \frac{2}{\alpha_1 - 2}} \lt(c_1 \lt(\frac{\alpha_1}2 - 1\rt)\rt)^{\frac2{\alpha_1-2}}.
\]
We then combine the above and \eqref{est_mom} to get
\[
\frac12\frac{d^2}{dt^2}\intrr |x|^2f\,dxdv \leq 2\tilde E(0) + \lt( \sum_{\alpha_i < 2}c_i \lt(1 - \frac{\alpha_i}2\rt)\rt)^{1 - \frac{2}{\alpha_1 - 2}} \lt(c_1 \lt(\frac{\alpha_1}2 - 1\rt)\rt)^{\frac2{\alpha_1-2}}.
\]
This shows that if the right hand side of the above is negative, i.e.
\[
\intrr |v|^2 f_0 \,dxdv + \lt( \sum_{\alpha_i < 2}c_i \lt(1 - \frac{\alpha_i}2\rt)\rt)^{1 - \frac{2}{\alpha_1 - 2}} \lt(c_1 \lt(\frac{\alpha_1}2 - 1\rt)\rt)^{\frac2{\alpha_1-2}} < \intr \rho_0 K\star\rho_0 \,dx,
\]
then the momentum of inertia should be negative in a finite time. Hence the life-span $T$ of the solution should be finite. 
\end{proof}

We next focus on the case with the linear Fokker--Planck operator. In this case, the total energy includes the entropy term. 

\begin{lemma}\label{lem_mom2} Let $f$ be a solution to the system \eqref{eq:VR} satisfying $f \in X(T)$. Then we have
\[
I''(t) \leq 2(1+\delta)E(0) + C_0 - 2(1+\delta)\intrr f \ln f\chi_{0 \leq f \leq 1}\,dxdv - \sigma I'(t) - \delta \intrr |v|^2 f\,dxdv
\]
for some $\delta > 0$ satisfying $2(1+\delta) < \max_{i=1,\dots, N} \alpha_i$, where $C_0 > 0$ is given by
\[
C_0 = \lt( \sum_{\alpha_i < 2(1+\delta)}c_i \lt(1 + \delta - \frac{\alpha_i}2\rt)\rt)^{1 - \frac{2}{\alpha_1 - 2}} \lt(c_1 \lt(\frac{\alpha_1}2 - 1 - \delta\rt)\rt)^{\frac2{\alpha_1-2}}.
\]
\end{lemma}

\begin{proof}
By Lemmas \ref{lem_energy} and \ref{lem_mom}, we deduce
\begin{align*}
I''(t) &= (1+\delta) \intrr |v|^2 f\,dxdv -\frac12\sum_{i=1}^Nc_i \alpha_i\int_{\R^d} \rho K_i \star \rho\,dx - \sigma I'(t) - \delta \intrr |v|^2 f\,dxdv\cr
&=2(1+\delta)E(t) + \sum_{i=1}^Nc_i \lt(1+\delta - \frac{\alpha_i}2\rt)\intr \rho K_i\star\rho\,dx - 2(1+\delta)\intrr f \ln f\,dxdv\cr
&\quad - \sigma I'(t) - \delta \intrr |v|^2 f\,dxdv,
\end{align*}
where $\delta >0$ is chosen such that $2(1+\delta) < \max_{i=1,\dots, N} \alpha_i$. Then similarly as in the proof of Theorem \ref{thm:blow} in the case $\sigma = 0$, we estimate
\begin{align*}
\sum_{i=1}^N c_i \lt(1 +\delta - \frac{\alpha_i}2\rt)\intr \rho K_i\star\rho\,dx &\leq \lt( \sum_{\alpha_i < 2(1+\delta)}c_i \lt(1 + \delta - \frac{\alpha_i}2\rt)\rt)^{1 - \frac{2}{\alpha_1 - 2}} \lt(c_1 \lt(\frac{\alpha_1}2 - 1 - \delta\rt)\rt)^{\frac2{\alpha_1-2}} \cr
&= C_0.
\end{align*}
Here we again set $\alpha_1 = \max_{i=1,\dots, N} \alpha_i$ without loss of generality. Thus we obtain
\begin{align*}
I''(t) &\leq 2(1+\delta)E(0) + C_0 - 2(1+\delta)\intrr f  (\ln f) \chi_{0 \leq f \leq 1}\,dxdv - \sigma I'(t) - \delta \intrr |v|^2 f\,dxdv.
\end{align*}
\end{proof}

In order to handle the entropy term appeared in Lemma \ref{lem_mom2}, we modify the classical lemma \cite{CIP94} to suit our methodology.
\begin{lemma}\label{lem_log} For given $\delta > 0$, there exists $C_\delta >0$, independent of $f \geq 0$, such that 
\[
- 2(1+\delta)\intrr f  (\ln f) \chi_{0 \leq f \leq 1}\,dxdv \leq C_\delta(I+1) + \delta \intrr |v|^2 f\,dxdv, 
\]
where $I = \frac12\intrr |x|^2 f\,dxdv$ and the constant $C_\delta$ is given by
\[
C_\delta = 4(1+\delta)(1 + \delta^{-1})^{\frac d{2+d}} (e^{-1} 2^{3d} \pi^{2d})^{\frac1{2+d}}.
\]
\end{lemma}
\begin{proof}Note that the following holds for any $s, \sigma \geq 0$:
\begin{align*}
-s\ln(s) \chi_{0 \leq s \leq 1} &= -s\ln(s) \chi_{e^{-\sigma} \leq s \leq 1} - s\ln(s) \chi_{e^{-\sigma} \geq s}\cr
&\leq s\sigma + 2e^{-2}\sqrt s \chi_{e^{-\sigma} \geq s}\cr
&\leq s\sigma + 2e^{-2} e^{-\sigma/2},
\end{align*}
where we used 
\[
\sup_{s \in (0,1)} -s\ln(s) \leq 2e^{-2}.
\]
We then take 
\[
s = f \quad \mbox{and} \quad \sigma = |x|^2 \epsilon_1 + |v|^2\epsilon_2,
\] 
where $\epsilon_i > 0$, $i=1,2$ will be determined later. Thus, we have
\begin{align*}
& - 2(1+\delta)\intrr f  (\ln f) \chi_{0 \leq f \leq 1}\,dxdv \cr
&\quad \leq 2(1+\delta)\intrr \lt(|x|^2 \epsilon_1 + |v|^2\epsilon_2 \rt)f\,dxdv + 4e^{-2}(1 + \delta)\intrr \exp\lt(- \frac{|x|^2 \epsilon_1 + |v|^2\epsilon_2}2 \rt)dxdv\cr
&\quad = 4(1+\delta)\epsilon_1 I(t) + 2(1+\delta)\epsilon_2\intrr |v|^2 f\,dxdv + 4e^{-2}(1 + \delta)\frac{(2\pi)^d}{(\epsilon_1 \epsilon_2)^{d/2}}.
\end{align*}
We finally choose $\epsilon_i>0$, $i=1,2$ so that 
\[
2(1+\delta)\epsilon_2 = \delta \quad \mbox{and} \quad 4(1+\delta)\epsilon_1 = 4e^{-2}(1 + \delta)\frac{(2\pi)^d}{(\epsilon_1 \epsilon_2)^{d/2}}
\]
to conclude the desired result.
\end{proof}

By combining Lemmas \ref{lem_mom} and \ref{lem_log}, we estimate 
\[
I''(t) \leq 2(1+\delta)E(0) + C_0 + C_\delta + C_\delta I(t) - \sigma I'(t). 
\]
Our main strategy is to show that $I(t)$ can be negative in a finite time under certain assumption on the initial data, and thus it leads to a contradiction. For this, we need to have the following Gr\"onwall-type lemma.
\begin{lemma}\label{lem_diff} Let $h = h(t)$ be a nonnegative $\calC^2$-function satisfying the following differential inequality:
\bq\label{eq_h}
h''(t) + c_1 h'(t) \leq c_2 h(t) + c_3, \quad h(0) = h_0, \quad h'(0) = h'_0
\eq
for some $c_i >0$, $i=1,2$ and $c_3 \in \R$. Then we have
\begin{align*}
h(t) &\leq \lt( h_0 + \frac{c_3}{\beta(\beta + c_1)} + \frac{1}{c_1 + 2\beta}\lt(h'_0 - \beta h_0 - \frac{c_3}{\beta + c_1}\rt) \rt)e^{\beta t}\cr
&\quad - \frac{1}{c_1 + 2\beta}\lt(h'_0 - \beta h_0 - \frac{c_3}{\beta + c_1} \rt) e^{-(c_1 + \beta)t} - \frac{c_3}{\beta (c_1 + \beta)},
\end{align*}
where $\beta > 0$ is given by
\bq\label{beta}
\beta := \frac{-c_1 + \sqrt{c_1^2 + 4c_2}}{2}.
\eq
\end{lemma}
\begin{proof} Set
\[
\tilde h(t) := h'(t) - \beta h(t) = e^{\beta t} \lt(h(t) e^{-\beta t}\rt)',
\]
then it follows from \eqref{eq_h} that 
\[
\tilde h' = h'' - \beta h' \leq -(c_1 + \beta)\lt(h' - \frac{c_2}{c_1 + \beta}h\rt) + c_3.
\]
Since $\beta$ given by \eqref{beta} satisfies 
\[
\beta = \frac{c_2}{c_1 + \beta},
\]
we obtain
\[
\tilde h' \leq -(c_1 + \beta)\tilde h + c_3.
\]
Solving the above implies
\[
e^{\beta t} \lt(h(t) e^{-\beta t}\rt)' = \tilde h(t) \leq \tilde h(0) e^{-(c_1 + \beta)t} + \frac{c_3}{c_1 + \beta}\lt(1 - e^{-(c_1 + \beta)t} \rt),
\]
and again solving the resulting differential inequality yields the desired result.
\end{proof}

\begin{proof}[Proof of Theorem \ref{thm:blow} in the case $\sigma > 0$]  We use Lemma \ref{lem_diff} with $h = I$, $c_1 = \sigma > 0$, $c_2 = C_\delta > 0$, and $c_3 = 2(1+\delta)E(0) + C_0 + C_\delta$ to obtain
\begin{align*}
I(t) &\leq \lt( I(0) + \frac{2(1+\delta)E(0) + C_0 +C_\delta}{\beta(\beta + \sigma)} + \frac{1}{\sigma + 2\beta}\lt(I'(0) - \beta I(0) - \frac{2(1+\delta)E(0)+ C_0 + C_\delta}{\beta + \sigma}\rt) \rt)e^{\beta t}\cr
&\quad - \frac{1}{\sigma + 2\beta}\lt(I'(0) - \beta I(0) - \frac{2(1+\delta)E(0)+ C_0+C_\delta}{\beta + \sigma} \rt) e^{-(\sigma + \beta)t} - \frac{2(1+\delta)E(0) + C_0 + C_\delta}{\beta (\sigma + \beta)},
\end{align*}
where
\[
\beta = \frac{-\sigma + \sqrt{\sigma^2 + 4C_\delta}}{2}.
\]
Since $\beta>0$, this implies that if
\[
 I(0) + \frac{2(1+\delta)E(0) + C_0 +C_\delta}{\beta(\beta + \sigma)} + \frac{1}{\sigma + 2\beta}\lt(I'(0) - \beta I(0) - \frac{2(1+\delta)E(0) + C_0 + C_\delta}{\beta + \sigma}\rt) < 0,
\]
which is equivalent to
\[
2(1+\delta)E(0) + \beta I'(0) < -\beta(\sigma + \beta) I(0) - C_\delta - C_0,
\]
then the life-span $T$ of solutions should be finite. 
\end{proof}

\subsection{Further discussion: repulsive and attractive interactions}\label{subsec:mixed}


In this part, we show the singularity formation for the system \eqref{eq:VR} with the repulsive-attractive singular interaction potentials. More precisely, let us consider the potential $K$ given by
\bq\label{pot_K}
K(x) = \frac{1}{|x|^{\alpha_1}} - \frac{1}{|x|^{\alpha_2}}  =: K_a(x) + K_r(x),
\eq
where $0< \alpha_1 \neq \alpha_2 < d$.
In this case, we estimate
\[
\intr \rho x \cdot \nabla K\star\rho\,dx=-\frac{\alpha_1}2 \intr \rho K_a \star \rho\,dx + \frac{\alpha_2}2 \intr \rho K_r \star \rho\,dx.
\]
Thus, 
\begin{align*}
I''(t) &=2(1+\delta)E(t) + \lt(1+\delta - \frac{\alpha_1}2\rt)\intr \rho K_a\star\rho\,dx -  \lt(1+\delta - \frac{\alpha_2}2\rt)\intr \rho K_r\star\rho\,dx \cr
&\quad - 2(1+\delta)\intrr f \ln f\,dxdv - \sigma I'(t) - \delta \intrr |v|^2 f\,dxdv.
\end{align*}
Let us denote the sum of the second and third terms on the right hand side by $I_K$. 

Analogous to that of Theorem \ref{thm:blow}, we assume that the exponent $\alpha_1$ in the attractive potential satisfies $\alpha_1 > 2$. We then consider two cases:  $\alpha_2 < 2$ and $\alpha_2 \geq 2$.

In the first case, we can choose $\delta > 0$ such that $\frac{\alpha_1}2 - 1 \geq \delta > 0 > \frac{\alpha_2}2 - 1$. This implies $1+\delta - \frac{\alpha_2}2 \geq 0$ and $1+\delta - \frac{\alpha_1}2 \leq 0$. Thus, $I_K \leq 0$.

On the other hand, if $\alpha_2 \geq 2$, we further assume $\alpha_1 > \alpha_2$, then 
\begin{align*}
\intr \rho K_r\star\rho\,dx &= \lt(\int_{|x-y|\leq 1} + \int_{|x-y|\geq  1}\rt)  \rho(x) \frac{1}{|x-y|^{\alpha_2}} \rho(y)\,dxdy\cr
&\leq \int_{|x-y|\leq 1} \rho(x) \frac{1}{|x-y|^{\alpha_1}} \rho(y)\,dxdy + \|\rho\|_{L^1}^2\cr
&\leq \intr \rho K_a\star\rho\,dx + 1.
\end{align*}
Thus, 
\begin{align*}
I_K &\leq \lt(1+\delta - \frac{\alpha_1}2\rt)\intr \rho K_a\star\rho\,dx -  \lt(1+\delta - \frac{\alpha_2}2\rt)\lt(\intr \rho K_a\star\rho\,dx + 1\rt)\cr
&\leq -\frac{\alpha_1 - \alpha_2}{2} \intr \rho K_a\star\rho\,dx +\frac{\alpha_2}2 -1 -\delta\cr
&\leq \frac{\alpha_2}2 -1 -\delta.
\end{align*}
In summary, one can deduce the following result.
\begin{theorem}\label{thm:blow2} Let $T>0$ and $d \geq 3$, and let $f$ be a solution to the system \eqref{eq:VR} satisfying $f \in X(T)$. Consider the interaction potential $K$ given as \eqref{pot_K}.
\begin{itemize}
\item (Vlasov equation: $\sigma =0$) Suppose $\alpha_1 \geq \max\{2, \alpha_2\}$ and 
\[
\intrr |v|^2 f_0 \,dxdv < \intr \rho_0 K\star\rho_0 \,dx + \lt( \frac{\alpha_2}2 - 1\rt){\bf 1}_{\{\alpha_2 \geq 2\}}.
\]
\item (Vlasov--Fokker--Planck equation: $\sigma>0$) Suppose $\alpha_1 > \max\{2, \alpha_2\}$ and the initial total energy and radial-weighted momentum is sufficiently small compared to the initial momentum of inertia so that
\[
2(1+\delta)E(0) + \beta I'(0) < -\beta(\sigma + \beta) I(0) - \lt(\frac{\alpha_2}2 -1 -\delta\rt){\bf 1}_{\{\alpha_2 \geq 2\}}.
\]
\end{itemize}
Then the life-span $T$ of the solution should be finite.
\end{theorem}
\begin{remark}One may extend Theorem \ref{thm:blow2} to the case where the interaction potential $K$ is given by 
\[
K(x) = \sum_{i=1}^N \frac{c_i}{|x|^{\alpha_i}},  \quad \alpha_i \in (0,d), \ \ c_i \in \R, \quad i=1,\dots, N.
\]
\end{remark}

\subsection*{Acknowledgments}
YPC has been supported by NRF grant (No. 2017R1C1B2012918) and Yonsei University Research Fund of 2021-22-0301. IJJ has been supported by the New Faculty Startup Fund from Seoul National University and the National Research Foundation of Korea grant (No. 2019R1F1A1058486).

\end{document}